\documentclass[11pt]{amsart} 
\usepackage{amssymb}
\usepackage{amsmath}
\usepackage{latexsym}
\usepackage{amscd}
\usepackage{amsfonts}
\usepackage{pb-diagram}
\usepackage[all]{xypic}
\usepackage[mathcal,mathscr]{eucal}
\usepackage{amsthm}

\usepackage{url}
\usepackage{pdfsync}
\usepackage{graphics}
\usepackage{bbm}
\usepackage[pdftex,colorlinks]{hyperref}
\usepackage{pifont}
\usepackage{amsbsy}
\usepackage{epstopdf}
\usepackage{times}
\usepackage{epic,eepic}

\newtheorem{theorem}{Theorem}[section]
\newtheorem{prop}[theorem]{Proposition}
\newtheorem{lemma}[theorem]{Lemma}
\newtheorem{corollary}[theorem]{Corollary}

\theoremstyle{definition}
\newtheorem{definition}[theorem]{Definition}

\theoremstyle{lemma}

\theoremstyle{definition}

\newtheorem{remark}[theorem]{Remark}
\newtheorem{example}[theorem]{Example}

\numberwithin{equation}{section}

\DeclareMathOperator{\Dom}{Dom}
\DeclareMathOperator{\Codom}{Codom}

\DeclareMathOperator{\SM}{\mathcal{S}\mathcal{M}}

\DeclareMathOperator{\Lip}{Lip}

\DeclareMathOperator{\Ker}{Ker}

\DeclareMathOperator{\grad}{grad}

\DeclareMathOperator{\Area}{Area}
\DeclareMathOperator{\Crit}{Crit}

\DeclareMathOperator{\tr}{tr}

\DeclareMathOperator{\id}{id}

\DeclareMathOperator{\orient}{orient}

\DeclareMathOperator{\Imag}{Im}
\DeclareMathOperator{\dvol}{dvol}
\DeclareMathOperator{\vol}{Vol}

\DeclareMathOperator{\HT}{HT}

\DeclareMathOperator{\diver}{div}

\newcommand\bR{{\mathbb R}}

\newcommand{\mlH}{\mathcal{H}}

\newcommand{\ra}{\rightarrow}

\newcommand{\ovra}{\overrightarrow}

\begin{document}
\title{A note on the area and coarea formulas for general volume densities and some applications}
\author{Daniel Cibotaru} 
\address{Universidade Federal do Cear\'a, Fortaleza, CE, Brasil}
\email{daniel@mat.ufc.br}
\author{Jorge de Lira}
\address{Universidade Federal do Cear\'a, Fortaleza, CE, Brasil}
\email{jorge.h.lira@gmail.com}

\subjclass[2010]{Primary 58C35, 49Q15; Secondary 46T12}
\begin{abstract} We present the  area and coarea formulas for Lipschitz maps, valid for general volume densities. As applications we give a short, "euclidean" proof of the anisotropic Sobolev inequality and  describe an anisotropic tube formula for hypersurfaces in $\mathbb{R}^n$. A discussion about the first variation of the anisotropic area is also included.
\end{abstract}
\maketitle
\tableofcontents
\section{Introduction}
The area and coarea formulas are well-established tools in analysis and geometry.  They are far reaching generalizations of the classical change of variables formula and of the Fubini Theorem and were proved by Federer in \cite{Fe}.  One purpose of this note is to take another look at these formulas and  rewrite them for \emph{general} volume densities. While they are not too hard to derive, these theorems  involve a certain formalism for which we have not found any precise reference in the literature. To emphasize their potential use, we consider several analytic and geometric applications such as the Sobolev inequality, the tube formula  for hypersurfaces and the first variation of the area, all of them in the \emph{anisotropic} world. The anisotropic tube formula seems new, while the anisotropic Sobolev inequality has received other proofs, notably one proof  given by Gromov (\cite{MS}) using methods of optimal transport.  The general coarea formula presented in this article allows one essentially to  copy-paste the classical proof  of Federer and Fleming (\cite{FF}) for  the euclidean Sobolev inequality.

The study of the first variation of the anisotropic area has at least two different approaches: the one by Andrews \cite{An} and the one by  Palmer \cite {P} and Koiso \cite{PK} and collaborators. Our  point of view, which is to make use of inner products as little as possible, is closer to Andrews line of investigation. We make a comparison between the two approaches at the end of this article. All these applications that we give are in our view interrelated and we included them all in order to better emphasize the philosophy we mentioned before. 
  
 One ingredient in the proof of the Sobolev inequality is  the isoperimetric inequality which in turn relies on the Brunn-Minkowski inequality. The reader can find a complete proof of the \emph {anisotropic} isoperimetric inequality (the so-called Wulff Theorem) in the context of integral currents in \cite{Ta}. We reproduce here the main idea of this proof, without any reference to currents whatsoever.  Along the way, we found it worthwhile to include  a \emph{geometric proof} of a result that puts the equal sign between the anisotropic outer Minkowski content and the anisotropic area  of the boundary of a $C^2$ manifold (Theorem \ref{anisoMink}). This was for two reasons: one because the  generalized change of variables (Proposition \ref{cov})  shows up and second, because this  lead us immediately to an anisotropic tube formula for hypersurfaces at least when the Wulff shape has a $C^2$ support function.  This result  involves the expected ingredients, namely the symmetric polynomials in the eigenvalues of an appropriately defined  \emph{anisotropic} Weingarten map. To our knowledge, such a tube formula has not appeared yet in the literature.  Inspired by this result we introduce the anisotropic connection and show that the divergence it induces is the right factor in the first variation of the anisotropic area, recovering thus a result of Andrews in \cite{An}. At last, we digress on how this relates to the work of Palmer and Koiso and Palmer (\cite{PK,P}).

As we said, the results presented in this note have some overlap with results published elsewhere. Using fine methods of geometric measure theory the equality of the anisotropic Minkowski content and the anisotropic area was recently proved by Chambolle,  Lisini and Lussardi (see \cite{CLL}) under much weaker regularity  conditions.  The anisotropic Sobolev inequality was first proved by Gromov using ideas of optimal transport (see the Appendix of \cite{MS}). This was vastly generalized by Cordero-Erausquin, Nazaret and Villani in \cite{CENV} using again  optimal transport.

 A coarea formula for real valued maps in Finsler geometry was proven by Shen in \cite{Shen}.  We obtain Shen's theorem  as a corollary of the main result (see Corollary \ref{funcman} and Remark \ref{funcmanrem}).   In their study of rectifiable subsets of metric spaces and Lipschitz maps between them,  Ambrosio and Kirchheim in \cite{AK} present some area and coarea formulas in which the underlying measures  are the Hausdorff measures.  In the reversible Finsler context, the Hausdorff measure is induced by the Busemann volume density.  We work instead with general volume densities.   It turns out that the relevant jacobians and cojacobians are given by relatively simple \emph{algebraic} expressions that involve only the differential of the map and the underlying norms/densities, very much in the spirit of the analogous euclidean objects.   These should be contrasted with the  coarea factors introduced by Ambrosio and Kirchheim which have a less direct definition, a reflection of the fact that they are meant to work in the general context of metric spaces. Here we are interested in manifolds and/or rectifiable sets.

One paper that motivated and influenced us  is the beautiful  article of \'Alvarez-Paiva and Thompson \cite{AT}. One finds there a quite extensive description of the important volume densities in Finsler geometry. The volume densities that appear here in the general area and coarea formulas do not need to arise in some "functorial" way   as  happens for example in the context  of Riemannian  or  Finsler geometry (see  \cite{AT})  in order for the theorems to hold. However, the search for such general formulas was triggered by considering the different notions of volume in Finsler geometry.

The article is structured as follows. We start by presenting the linear picture. The main novelty  is the notion of a \emph{codensity} that is used in the definition of the cojacobian. This is a density on the dual vector space obtained, in a certain sense, by taking the quotient of a  density of complementary dimension on the original vector space and a top-degree density. We next describe the change of variables and the Fubini formula for densities in the smooth ($C^1$) framework. In Section \ref{sec3}, we present the  area and coarea formulas for Lispchitz maps between the euclidean spaces.   We should emphasize  that special attention is paid to orientation issues  whenever the densities involved are not symmetric.

 The proofs of the area and coarea theorems for general densities use their well-known Riemannian analogues of \cite{F}  and some canonical relations between the correction factors derived in Section \ref{sec1}. The whole Section \ref{sec4} is dedicated to the proof of the anisotropic Sobolev inequality together with the aforementioned result about anisotropic Minkowski content. In Section \ref{sec5} we derive the tube formula and inspired by the proof we introduce a natural anisotropic connection on a hypersurface and we show that the first variation of the (anisotropic) area is the integral of the  divergence with respect to this connection. We include a version of the divergence theorem relevant to this framework.

\section{The linear picture}\label{sec1}
Let $V$ be a real vector space of dimension $n$. The cone of simple $k$-multivectors  is the subset $\Lambda^k_sV$ of elements of $\Lambda^kV$  which can be written as a wedge product of $k$-vectors from $V$.
\begin{definition}  A $k$-(volume) density on $V$ is a map $F:\Lambda^k_sV\ra \mathbb{R}_{\geq 0}$ which takes the value $0$ only in $0$ and is homogeneous of degree $1$, i.e.
\[ F(a\xi)=aF(\xi), \quad\quad\forall a\geq 0,\xi\in \Lambda^k_sV
\]
Denote by $\mathscr{D}_k(V)$ the space of $k$-volume densities. We will let $\mathscr{D}_k^+(V)$ denote the symmetric $k$-volume densities, meaning those that also satisfy
\[ F(\xi)=F(-\xi).
\]
\end{definition}
\begin{example} An inner product on $V$ induces norms on all $\Lambda^kV$ and the corresponding norms give volume densities in all degrees.
\end{example}

\begin{example} An $n$-volume density is just a norm on the line $\Lambda^n V$. For this reason we will occasionally use the  symbol $\|\cdot\|$ to denote a $k$-density, even if this might not necessarily arise from a norm on the whole $\Lambda^kV$. 

 Notice that  every $n$-form $\Omega:\Lambda^nV\ra \mathbb{R}$ gives rise to a symmetric $n$-volume density by letting $|\Omega|:=|\cdot|\circ \Omega$. Conversely, a symmetric $n$-density and an orientation on $V$ give rise to an $n$-form.
\end{example}

For the most part in this note we will work with symmetric volume densities.  In order to work with asymmetric volume densities there is one limitation one has to impose  and this  is that one should be working in the category of \emph{oriented} vector spaces. The reason is simple. Take for example a curve in a Finsler manifold. If the Finsler norms are not symmetric then one ends up having a forward length and a possibly different backward length of the same curve, depending on whether one is going along the curve in one direction or in the opposite direction. So one cannot hope to get a change of variables formula for maps that reverse the orientation. This is the same condition one has when building the theory of integration for differential forms: the background manifolds are supposed to be oriented. Most of the definitions and results are stated for symmetric densities, while in subsequent remarks we clarify what are the modifications in the oriented/non-symmetric case.

\begin{definition}\label{jacdef}
Let $V$ and $W$ be two vector spaces and $F\in\mathscr{D}_n^+(V)$ and $G\in\mathscr{D}_n^+(W)$  two symmetric $n$-densities. Suppose $\dim{V}=n$ and let $A:V\ra W$ be a linear map. The jacobian of $A$ with respect to the densities $F$ and $G$ is
\[ J(A)=J(A; F,G)=G(\Lambda^nA(\xi)),
\]
where $\xi\in\Lambda^n V$ is one of the two vectors that satisfies $F(\xi)=1$.  If  $V$ and $W$ are    oriented  then $F$ and $G$ are allowed to be non-symmetric and the definition of the jacobian is modified by requiring that $\xi$ be positively oriented.
\end{definition}

We make two trivial remarks. The jacobian is non-zero only if  $A$ is injective. If $G$ is the restriction of a norm on $\Lambda^nW$ then the jacobian is the norm of the operator $\Lambda^nA$.

\begin{example} If $V$ and $W$ are endowed with inner products that induce the densities $F$ and $G$ respectively then
\[ J(A)=\sqrt{\det{A^*A}}.
\] 
One checks this by considering $A$ to be an isomorphism onto its image.
\end{example}

\begin{example} The Holmes-Thompson volume density of a normed space $V$ is defined by (see \cite{AT})
\begin{equation}\label{HT} \mu(v_1\wedge\ldots\wedge v_n)=\epsilon_n^{-1}\int_{B(V^*)} ~dv_1^*\ldots dv_n^*,
\end{equation}
where $B(V^*)$ is the dual unit ball, $v_1^*,\ldots, v_n^*$ is the dual basis to $v_1,\ldots, v_n$ and $\epsilon_n$ is the volume of the euclidean unit ball. The notation $dv_1^*\ldots dv_n^*$ is a substitute for the Lebesgue measure induced by the basis $\{v_1^*\ldots v_n^*\}$. One can rewrite (\ref{HT}) as:
\[   \mu(v_1\wedge\ldots\wedge v_n)=\epsilon_n^{-1}\int_{B(V^*)} ~dv_1\wedge\ldots\wedge dv_n,
\] 
where now, $dv_1\wedge\ldots\wedge dv_n$ is an element of $V^{**}$  corresponding to $v_1\wedge\ldots\wedge v_n$ under the canonical isomorphism $V\simeq V^{**}$ and the orientation on $V^*$ is the one given by $v_1^*\wedge\ldots\wedge v_n^*$.

 The jacobian of a bijective linear  map $A:V\ra W$ between two normed vector spaces endowed with the Holmes-Thompson volumes is:
\[ J_{\HT}(A)=\left|\frac{\int_{B(W^*)} \Lambda^nA (\xi)}{\int_{B(V^*)}\xi}\right|,
\]
where $\xi$ is any non-zero element of $\Lambda^n V$, considered as an $n$-form over $V^*$. If $W=V$, possibly with a different norm, then this simplifies to:
\[ J_{\HT}(A)=|\det{A}|\frac{\int_{B(W^*)}\xi}{\int_{B(V^*)}\xi}=|\det{A}|\frac{\vol{(B(W^*))}}{\vol{(B(V^*))}}.
\]
\end{example}
\vspace{0.3cm}
In order to define the correction factor used in the coarea formula we need the next elementary result.
\begin{lemma}\label{codens} Let $V$ be a vector space of dimension $n+m$ and let $F\in\mathscr{D}_n^+(V)$ and $\mu\in\mathscr{D}_{n+m}^+(V)$.  Let $v_1^*,\ldots v_m^*$ be linearly independent vectors and choose $w_1^*,\ldots, w_n^*$ such that the set $\{v_1^*\ldots v_{m}^*,w_1^*,\ldots w_n^*\}$ forms  a basis of $V^*$ with dual  basis $ v_1,\ldots v_m, w_1,\ldots, w_n$ of $V$. Then
\[ F^*_{\mu}(v_1^*\wedge\ldots \wedge v_m^*):=\frac{F(w_1\wedge\ldots \wedge w_n)}{ \mu( v_1\wedge \ldots\wedge v_m \wedge w_1\wedge\ldots \wedge w_n)}
\]
is independent of the choice of $w_1^*,\ldots, w_n^*$ and thus induces a well defined $m$-volume density on $V^*$. 
\end{lemma}
\begin{proof}  Let $U^*:=\langle v_1^*,\ldots, v_m^*\rangle\subset V^*$ with dual space $U$. We obviously have an exact sequence:
\[ 0\ra W\ra V\ra U\ra 0,
\]
where $W$ is the kernel of the projection $V\ra U$. The $n$-dimensional subspace $W$ is endowed with an $n$- density by restricting  $F$. Together with the $n+m$ density on $V$ one gets an $m$-density $H$ on $U$: let $u_1,\ldots ,u_m$ be a basis of $U$ and choose  some lifts of the $u$'s  in $V$, call them  $e_1,\ldots ,e_m$  and a basis $\{f_1,\ldots f_n\}$ of $W$. Then
\[ H(u_1\wedge\ldots \wedge u_m):=\frac{\mu(e_1\wedge \ldots \wedge e_m\wedge f_1\wedge\ldots f_n)}{F(f_1\wedge\ldots \wedge f_n)}.
\]
This definition does not depend on the basis $f_1,\ldots, f_n$ of $W$ or on the lifts $e_1,\ldots e_m$. Indeed, if one makes  different choices $\{e_1,'\ldots,e_m'\}$  and $\{f_1',\ldots,f_n'\}$ then the change of basis matrix going from $\{e_1\ldots e_m, f_1,\ldots f_n\}$ to $\{e_1'\ldots e_m', f_1',\ldots f_n'\}$   has the block decomposition:
\[ X=\left(\begin{array}{cc} A & B \\ C&D\end{array}\right) \quad\mbox{ with } \quad A=1 \mbox{ and } C=0
\]
with $D$ the matrix obtained by going from $\{f_1,\ldots, f_n\}$  to $\{f_1',\ldots, f_n'\}$. Therefore $\det{X}=\det{D}$. 

From $H$ we get an $n$-density on $U^*$:
\[H^*(u_1^*\wedge\ldots \wedge u_m^*):=\frac{1}{H (u_1\wedge\ldots \wedge u_m)}
\]
Notice now that   $W=\langle w_1,\ldots, w_n \rangle$ and that $v_1, \ldots v_m\in V$ are lifts of the vectors in $U$ which represent the dual basis to $v_1^*,\ldots, v_m^*$. Thus $H^*=F^*_{\mu}\bigr|_{U^*}$.
\end{proof}
\begin{definition} The $m$-densities on $V^*$  resulting by taking the "quotient" of an $n$-density $F$ on $V$  and an $(n+m)$-density $\mu$ as in  Lemma \ref{codens} will be called $m$-codensities on $V$.

If $\dim{V}=m$ then the codensity associated to $\mu\in\mathscr{D}_m^+(V)$ is 
\[ \mu^*(v_1^*\wedge\ldots \wedge v_{m}^*)=\frac{1}{\mu(v_1\wedge\ldots\wedge v_{m})}.
 \]
 The space of all $m$-codensities on $V$ corresponding to the same top degree volume density $\mu$ will be denoted by $\mathscr{D}^m_{\mu}(V)$. 
\end{definition}
\begin{remark}\label{Fns} We will allow $F$ to be non-symmetric while keeping $\mu$ symmetric but in that case $V$ will be oriented and the basis $v_1,\ldots, v_m,$ $w_1,\ldots w_n$ is required to be positively oriented in the Definition-Lemma \ref{codens} which goes through with some obvious modifications.
\end{remark}
\begin{remark}\label{top dens}   A codensity should not be confused with  a dual density. Given $F\in\mathscr{D}_n(V)$ one defines $F^{\sharp}\in\mathscr{D}_n(V^*)$ by putting
\[ F(\omega)^{\sharp}:=\sup_{\stackrel{F(\xi)=1}{\xi\mbox{ simple}}}|\omega(\xi)|
\]

They are the same though in top dimension, i.e. if $n=\dim{V}$ and $\mu\in\mathscr{D}_n(V)$, then $\mu^{\sharp}=\mu^*$.
  \qed
\end{remark} 

It is useful to have another description of codensities. Fix a  symmetric density $\mu$ of top degree $n+m$ and let $\Omega$ be a top degree form on $V$ such that 
\[ \mu=|\Omega|.
\]

Notice that $\Omega$ induces a natural Hodge isomorphism of vector spaces.
\begin{equation}\label{iota} \iota_{\Omega}:\Lambda^nV\ra\Lambda^{m}V^*,\quad\quad\quad\iota_{\Omega}(\theta)=\{\eta\ra\Omega(\theta\wedge \eta)\}.
\end{equation}
This  coincides with the classical Hodge duality, once one picks an inner product on $V$ whose volume form coincides with $\Omega$ and identifies $\Lambda^m V^*$ with $\Lambda^m V$ via the inner product.

We will assume that $\Omega$ is positively oriented if $V$ comes with an orientation.

For reasons which have to do with orientation conventions   we use
\[\iota_{\Omega}^{*}:\Lambda^nV\ra\Lambda^{m}V^*, \quad\quad\quad\iota_{\Omega}^*(\theta)=\{\eta\ra\Omega(\eta\wedge \theta)\}.
\] 
Notice that $\iota_{\Omega}$ and $\iota_{\Omega}^{*}$ are also  isomorphisms between the cones of simple vectors. Indeed, if $w_1\wedge\ldots\wedge w_n\in\Lambda^nV$ then  choose $v_1,\ldots ,v_m\in V$ such that
\begin{equation}\label{orbasis} \Omega( v_1\ldots\wedge v_m\wedge w_1\wedge\ldots\wedge w_n)= 1
\end{equation}
It is then straightforward to check that $\iota_{\Omega}^*(w_1\wedge\ldots \wedge w_n)= v_1^*\wedge\ldots\wedge v_m^*$, where $v_1^*,\ldots, v_m^*$ are those vectors in the dual  basis of $\{w_1,\ldots ,w_n,v_1,\ldots , v_m\}$ which vanish on $\langle w_1,\ldots w_n\rangle$. On the other hand,
 \[\iota_{\Omega}(w_1\wedge\ldots\wedge w_n)=(-1)^{mn}v_1^*\wedge\ldots\wedge v_m^*.\]
 
 Notice that  in the presence of an orientation, it makes more sense to write $F^*_{\Omega}$ for the codensity resulting from a density $F$, as discussed in Remark \ref{Fns}. We will keep the notation $F^*_{|\Omega|}$ to suggest that no orientation is present. 

\begin{prop} \label{altcodens} Let $F\in\mathscr{D}_n^+(V)$ and let $F_{|\Omega|}^*$ be the induced $m$-codensity. Then
\[ F_{|\Omega|}^*=F\circ( \iota_{\Omega}^*)^{-1}.
\]
If $V$ is oriented then
\[ F_{\Omega}^*=F\circ(\iota_{\Omega}^*)^{-1}.
\]
\end{prop}
\begin{proof} Using the notation preceding the proposition we check that $F=F_{|\Omega|}\circ \iota_{\Omega}^*$:
\[F_{|\Omega|}^*(\iota_{\Omega}^*(w_1\wedge\ldots\wedge w_n))=F_{|\Omega|}^*(v_1^*\wedge\ldots\wedge v_m^*)=\frac{F(w_1\wedge\ldots\wedge w_n)}{|\Omega|( v_1\ldots\wedge v_m\wedge w_1\wedge\ldots\wedge w_n)}\]
and the result follows by (\ref{orbasis}). The same proof works for the second part.
\end{proof}
\begin{remark} Notice that for a non-symmetric $F$ it makes all the difference that one uses $\iota_{\Omega}^*$ instead of $\iota_{\Omega}$. Obviously, all this is  related to the choice of the order in which to write $v_1,\ldots, v_m,w_1,\ldots w_n$ in  Lemma \ref{codens}. This choice of order in turn is  related  to considerations concerning integration over the fiber and Fubini formula for forms in which the orientation on the "top" space is the orientation of the base space "tensored" with the orientation on the fiber, the so-called fiber-last convention. 
\end{remark}

\begin{definition} Let $V$ and $W$ be two vector spaces of dimension $n+m$ and $m$ respectively.  Suppose that $V$ is endowed with a $F\in\mathscr{D}_n^+(V)$ and $\mu\in\mathscr{D}_{n+m}^+(V)$, while $W$ has  $\nu\in\mathscr{D}_m^+(V)$. Let $A:V\ra W$ be a linear map. The cojacobian of $A$ with respect to $F,\nu$ and $\mu$ is:
\[ C(A)=C(A;\nu^*,F_{\mu}^*):=F_{\mu}^*(\Lambda^nA^*(\omega))=J(A^*; \nu^*,F_{\mu}^*), 
\]
where $\omega\in\Lambda^m W^*$ is one of the two vectors of norm $1$, i.e $\nu^*(\omega)=1$. 
\end{definition}

When $V$ and $W$ have inner products and all densities involved come from these then one recovers the usual definition of the cojacobian:
\[ C(A)=\sqrt{\det{AA^*}}.
\]

\begin{remark} If $V$ and $W$ are oriented then we will allow $F$ to be non-symmetric but $\mu$ and $\nu$ are kept symmetric. In that case the definition of the cojacobian is modified by requiring that $\omega$ be positively oriented and keeping in mind the different definition for $F^*_{\mu}$ (Remark \ref{Fns}).
\end{remark}

\begin{example}\label{excodens1}  Notice first that the cojacobian is zero when $A$ is not surjective.

Suppose therefore that $A$ is surjective and choose $w_1,\ldots, w_n$ a basis of $\Ker A$. Let $v_1,\ldots, v_m\in V$ be vectors such that $Av_1\ldots ,Av_m$ is a basis in $W$. Then 
\[C(A;\nu^*,F_{\mu}^*)=\frac{F(w_1\wedge\ldots \wedge w_m)\cdot\nu(Av_1\wedge\ldots\wedge Av_n)}{\mu(v_1\wedge\ldots \wedge v_n\wedge w_1\wedge\ldots w_m)}.
\] 
When both spaces are  oriented then $v_1,\ldots,v_m$ are chosen such that $Av_1\ldots ,Av_m$ form a positively oriented basis of $W$ and  $w_1,\ldots w_n$ such that $v_1,\ldots,v_m,w_1\ldots w_n$ form a positively oriented basis of $V$. This is the same thing as saying that $w_1,\ldots,w_n$ form an oriented basis of the "fiber" $\Ker A$ where the fiber is given the induced orientation.
\end{example}
\begin{example}\label{ex2} In Section \ref{sec4} we will treat the case $m=1$  in more detail. We will be concerned with an oriented  vector space $V$ of dimension $n+1$ with a volume form $\Omega:V\ra\ \bR$ and a codimension $1$-convex density, meaning a function $F:\Lambda^{n}V\ra \mathbb{R}$ such that $F$ is positively homogeneous and convex, also called an anisotropic functional.  One  defines the convex dual functional as in Prop. \ref{altcodens}) $\bar{F}^*_{\Omega}:V^*\ra \mathbb{R}$   and $\bar{F}:V\ra \bR$ by duality:
\[ \bar{F}(v)=\sup_{\bar{F}^*_{\Omega}(f)\leq 1} f(v).
\]
One checks easily that
\[ C(f, dt, F_{\Omega}^*)=\bar{F}^*(f). 
\]
If $\bar{F}$ is strictly convex, one  defines the Legendre isomorphism $V^*\setminus\{0\}\ra V\setminus\{0\}$ by associating to $f$ the unique vector $v_f$ such that $\bar{F}^*(f)=\bar{F}(v_f)$ and
\[ f(v_f/\bar{F}(v_f))=\sup_{\bar{F}(w)=1} f(w).
\]
The vector $v_f$ is an example of anisotropic normal to the hyperplane $\Ker f$. When $f$ happens to be the differential of some $C^1$-function $g:V\ra \mathbb{R}$ at a point $p$, then $v_f$ is usually called the Finslerian gradient of $g$ at $p$ and is denoted $\nabla g(p)$. We therefore have
\[ C(dg, dt, F_{\Omega}^*)=\bar{F}^*(dg)=\bar{F}(\nabla g). 
\]
\end{example}

\vspace{0.3cm}
In order for the next result not to lead to some vacuous statements we fix some conventions. Thus the jacobian of a map with respect to two densities is well defined  only when the following dimensional condition is satisfied: 
 \[\deg{F}=\deg{G}=\dim{\Dom(A)}.\]
   It is convenient to set $J(A;F,G)=0$ if this does not hold. Similarly for the cojacobian $C(A;\nu^*,F_{\mu}^*)$ one needs 
   \[\deg{\nu}=\deg{\mu}-\deg{F}=\dim{\Codom{(A)}}.\]
    We set $C(A;\nu^*,F_{\mu}^*)=0$ otherwise. 
    
 The following straightforward properties of the jacobian and cojacobian are needed for the proof of the area and coarea formulas. We will next  that all densities are symmetric.
\begin{prop} \label{properties} Let $V,W,Z$ be three vector spaces and $A:V\ra W$ and $T:W\ra Z$ be linear maps. 
\begin{itemize}
\item[(a)] If $\dim V=\dim W=n$, $\mu\in\mathscr{D}_n^+(V)$ and $\nu\in\mathscr{D}_n^+(W)$ then
\[ J(A;\mu,\nu)=C(A;\nu^*, \mu^*)\]
\item[(b)] If  $\dim{V}=n$,  $\mu\in\mathscr{D}_n^+(V)$, $G\in\mathscr{D}_n^+(W)$, $H\in\mathscr{D}_n^+(Z)$   then
 \[ J(T\circ A; \mu,H)=J(A;\mu, G)\cdot J\left(T\bigr|_{\Imag{A}};G,H\right).
\]
\item[(c)] If $\dim{Z}=n$, $\lambda\in\mathscr{D}_n^+(Z)$,  $F_{\mu}^*\in\mathscr{D}^n_{\mu}(V)$, $G_{\nu}^*\in\mathscr{D}^n_{\nu}(W)$  then
  \[ C(T\circ A;\lambda^*,F_{\mu}^*)=C(T; \lambda^*,G_{\nu}^*)\cdot J\left(A^*\bigr|_{\Imag{T}^*}; G_{\nu}^*,F_{\mu}^*\right).
\]
\item[(d)] If $\dim W=\dim Z=n$, $\nu\in\mathscr{D}_n^+(W)$, $\lambda\in\mathscr{D}_n^+(Z)$,  $F_{\mu}^*\in\mathscr{D}^n_{\mu}(V)$  then:
\[ C(T\circ A;\lambda^*,F_{\mu}^*)=C(T; \lambda^*,\nu^*)\cdot C(A;\nu^*, F_{\mu}^*)=J(T;\nu, \lambda)\cdot C(A;\nu^*, F_{\mu}^*)
\]
\item[(e)] If  $\dim{V}=\dim{W}=n$,  $\dim{Z}=k$,  $\mu\in\mathscr{D}_{n}^+(V)$, $\nu\in\mathscr{D}_n^+(W)$, $F\in\mathscr{D}_{n-k}^+(V)$, $G\in\mathscr{D}_{n-k}(W)$, $\lambda\in\mathscr{D}_k(Z)$, $A$ is an isomorphism and $T$ is surjective then
\[ J(A\bigr|_{\Ker{T\circ A}}; F,G)\cdot C(T\circ A;\lambda^*,F_\mu^*)=C(T;\lambda^*,G_{\nu}^*)\cdot J(A;\mu,\nu),
\]
where by convention the jacobian of the unique map on the null vector space is set to be equal to $1$.
\end{itemize}
\end{prop}
\begin{proof} $(a)$ It is easy to check that $J(A)J(A^{-1})=1$. Now if $\{w_1,\ldots, w_n\}$ is a basis of $W$, then $\{A^*w_1^*,\ldots, A^*w_n^* \}$ and $\{A^{-1}w_1,\ldots, A^{-1}w_n\}$ are dual basis of $V^*$ and $V$ respectively. Clearly $\nu(w_1\wedge\ldots\wedge w_n)=1\Leftrightarrow\nu^*(w_1^*\wedge\ldots w_n^*)=1$. Therefore
\[ C(A)=\mu^*(A^*w_1^*\wedge\ldots\wedge A^*w_n^*)=\frac{1}{\mu(A^{-1}w_1\wedge\ldots \wedge A^{-1}w_n)}=\frac{1}{J(A^{-1})}=J(A).
\]

$(b)$ If $A$ is not injective both sides are zero. For $A$ injective just choose a basis.

$(c)$ If $T$ is not surjective both sides are zero. Either way, one applies $(b)$  to $T^*$ and $A^*$.

$(d)$ If $T$ is not an isomorphism then $T$ is not surjective and both sides are zero.  The result follows directly from $(c)$ and $(a)$.

$(e)$ This follows from $(c)$, $(a)$ and the next lemma applied to $U=\Ker{T}$. Notice that $\Imag{T}^*=U^{\perp}$.
\end{proof}

\begin{lemma} Let $A:V\ra W$ be an isomorphism of vector spaces and $U\subset W$ a vector subspace. Let $U^{\perp}:=\{\alpha\in W^*~|~\alpha\bigr|_U\equiv 0\}$. Suppose both $V,W$ are endowed with top degree densities and also with densities of degree equal to $\dim{U}$. Then
\[ J(A)=J\left(A\bigr|_{A^{-1}(U)}\right)J\left(A^*\bigr|_{U^{\perp}}\right),
\]
where $J(A^*\bigr|_{U^{\perp}})$ is computed using the induced $(\dim{W}-\dim{U})$-codensities on $V$ and $W$.
\end{lemma}
\begin{proof} Let $u_1^*,\ldots, u_k^*$ be a basis of $U^{\perp}$ which we complete to a basis $u_1^*,\ldots, u_n^*$ of $W^*$. It follows quickly that $U$ is the span of $u_{k+1},\ldots, u_n$. We will use the norm symbol $\|\cdot\|$ for all densities involved and $\|\cdot\|^*$ for codensities.

Clearly $\|u_1^*\wedge\ldots\wedge u_k^*\|^*=1$ is equivalent with $\|u_{k+1}\wedge\ldots\wedge u_n\|=\|u_1\wedge\ldots\wedge u_n\|$.

The dual basis to $\{A^*u_1^*,\ldots,A^*u_n^*\}$ is $\{A^{-1}u_1,\ldots,A^{-1}u_n\}$. Now
\[ J(A^*\bigr|_{U^{\perp}})=\|A^*u_1^*\wedge\ldots\wedge A^*u_k^*\|^*=\frac{\|A^{-1}u_{k+1}\wedge\ldots\wedge A^{-1}u_n\|}{\|A^{-1}u_{1}\wedge\ldots\wedge A^{-1}u_n\|}=
\]
\[=\frac{\|A^{-1}u_{k+1}\wedge\ldots\wedge A^{-1}u_n\|}{\|u_{k+1}\wedge\ldots\wedge u_n\|}\cdot\frac{\|u_1\wedge\ldots\wedge u_n\|}{\|A^{-1}u_1\wedge\ldots\wedge A^{-1}u_n\|}=\frac{J(A^{-1}\bigr|_{U})}{J(A^{-1})}=\frac{J(A)}{J(A\bigr|_{A^{-1}(U)})}.
\]
\end{proof}

For the conclusion of this section let us check  that the jacobian defined above coincides with the one introduced by Kirchheim (\cite{AK},\cite{K}) when the relevant densities  arise from the Busemann-Hausdorff definition of volume. 

\begin{definition} Let $V$ be a normed vector space of dimension $n$. Then the Busemann $n$-volume density on $V$ is
\[\mu_B:\Lambda^nV\ra\mathbb{R},\quad\quad \mu_B(v_1\wedge\ldots\wedge v_n):=\frac{\epsilon_n}{\int_B dv_1\ldots dv_n},
\]
where $\epsilon_n$ is the volume of the Euclidean unit ball and the integral in the denominator is the volume of the unit ball $B$ in $V$ measured with respect to the Lebesgue measure resulting by considering the basis $\{v_1,\ldots, v_n\}$ to be orthonormal.
\end{definition}
It is clear that one can define a $k$ -Busemann volume density on $V$ by considering $k$-dimensional subspaces of $V$ with the induced norm.
\begin{definition} Let $V$ be a normed vector space of dimension $n$. The K-jacobian of a linear map $A:V\ra W$ between normed spaces is 
\[ J^K(A):=\frac{\epsilon_n}{\mathcal{H}_n(\{x~|~\|Ax\|\leq 1\})},
\] 
where $\mathcal{H}_n$ is the $n$-Hausdorff measure on $V$ seen as a metric space with respect to its norm.
\end{definition}
Notice that if $A$ is not injective then the $K$-jacobian is $0$ as the denominator represents the volume of an infinite cylinder.
\begin{prop} The $K$-jacobian coincides with the jacobian of Definition \ref{jacdef} when the $n$-volume densities on $V$ and $W$ are the Busemann densities induced by the norms.
\end{prop}
\begin{proof} The Busemann $n$-density induces a translation invariant measure on $V$ (a multiple of any {\it non-canonical} Lebesgue measure on $V$) and a result of Busemann says that the volume of a measurable set $K$ with respect to this measure equals its Hausdorff measure $\mathcal{H}_n(K)$. Let $K:=\{x~|~\|Ax\|\leq 1\}$.

Let $v_1,\ldots, v_n$ be vectors such that $\mu_B(v_1\wedge\ldots\wedge v_n)=1$. Then
\begin{equation}\label{eq1} \mathcal{H}_n(K)=\int_K~dv_1\ldots dv_n,
\end{equation}
where $dv_1\ldots dv_n$ represents the Lebesgue measure $\lambda$ determined by $v_1,\ldots v_n$. Indeed if $\mu$ and $\nu$ are two translation invariant measures on $V$ then 
\[ \frac{\vol(K,\mu)}{\vol(P,\mu)}=\frac{\vol(K,\nu)}{\vol(P,\nu)}
\]
for any non-zero measure set $P$. If we let $P$ be the parallelotope spanned by $v_1,\ldots,v_n$, $\mu:=\mu_B$ and $\nu=\lambda$ we get (\ref{eq1}).

If $A:V\ra W$ is injective, then $A$ is an isometry between the euclidean space $V$ with orthonormal basis $v_1,\ldots, v_n$ and $A(V)$ with orthonormal basis $Av_1,\ldots, Av_n$. We therefore get
\[ \mathcal{H}_n(K)=\int_K~dv_1\ldots dv_n=\int_{A(K)}~dA(v_1)\wedge\ldots\wedge dA(v_n)=\]\[=\int_{A(V)\cap B}~dA(v_1)\wedge\ldots\wedge dA(v_n)=\frac{\epsilon_n}{\mu_B(A(v_1)\wedge\ldots\wedge A(v_n))}=\frac{\epsilon_n}{J(A)}.
\]
\end{proof}

\section{The manifold picture}\label{Sec2} In its simplest form, the area formula  is a change of variables  while the coarea is  Fubini's theorem. We state the analogues of these results for volume densities in the  manifold case.

\begin{definition}\label{defdens} Let $M$  be $C^1$-manifold of dimension $n$. A $k$-volume density on $M$ is a continuous map $F:\Lambda^kTM\ra\mathbb{R}$, such that  $F(m)\in \mathscr{D}_k(T_mM)$, for every $m\in M$. Another name for $F$ is  $k$-dimensional positive parametric integrand.
\end{definition} 
Again we will distinguish two situations, the unoriented one in which the densities are required to be symmetric and the oriented one in which the top degree densities are still required to be symmetric while the ones of lower degree could be non-symmetric.

A $k$-volume density restricted to a $C^1$-submanifold $M$ of dimension $k$,  determines a measure on this set via a standard process which we review. First let us assume that $k=n$ and that $M$ is an open subset of $\mathbb{R}^n$. Then for every Borel subset $B\subset M$ one defines
\begin{equation} \label{defF} \int_{B}~dF=\int_BF_x(e_1\wedge\ldots\wedge e_n)~d\lambda(x),
\end{equation}
where $e_1\wedge\ldots\wedge e_n$ is the vector of Euclidean norm $1$ (positively oriented) and $\lambda$ is the Lebesgue measure. Already in this definition we notice the appearance of the jacobian since
\[ F_x(e_1\wedge\ldots\wedge e_n)=J(\id_{T_x\mathbb{R}^n}; \lambda, F_x).
\]
One extends this definition to Borel subsets of an abstract manifold $M$ of dimension $k$, by putting
\[ \int_B~dF:=\int_{\alpha^{-1}(B)}~d\alpha^*F
\] where $\alpha:\Omega\ra M$ is a chart, $\Omega\subset\bR^n$ open, bounded and $B\subset \alpha(\Omega)$ is a Borel set. We denote by $\alpha^*F$ the pull-back of the density $F$ via the map $\alpha$. This definition is independent of the chart because of the standard change of variables formula. 

In other words, if $\phi:\Omega\ra \Omega$ is a diffeomorphism  with $\Omega\subset \bR^n$ open then for every Borel subset $\Gamma\subset \Omega$ we have
\[\int_{\Gamma}~d\phi^*F:=\int_{\Gamma} \phi^*F_x(e_1\wedge\ldots\wedge e_n)~d\lambda(x)=\int_{\Gamma}F_{\phi(x)}(d\phi_x(e_1)\wedge\ldots\wedge d\phi_x(e_n))~d\lambda(x)=\]\[=\int_{\Gamma}F_{\phi(x)}(\det{(d\phi_x)}(e_1\wedge\ldots \wedge e_n))=\int_{\Gamma}|\det{(d\phi_x)}|F_{\phi(x)}(e_1\wedge\ldots\wedge e_n)~d\lambda(x)=\]\[=\int_{\phi({\Gamma})}F_x(e_1\wedge\ldots\wedge e_n)~d\lambda(x)=\int_{\phi({\Gamma})}~dF.
\] 
Using the jacobian, we can rewrite this as 
\[ \int_B J(d\phi_x;\lambda,F_x)~d\lambda(x)=\int_{\phi(B)}~dF.
\]
\begin{remark} The discussion above was only for symmetric densities. If $F$ is non-symmetric \footnote{Non-symmetric means, as usual, possibly non-symmetric.} then $M$ is assumed oriented and the charts above are orientation preserving, following the spirit of the integral of differential forms.
\end{remark}
Without further ado let us state the change of variables.
\begin{prop}\label{cov} Let  $(B,F)$ and $(P,G)$ be two $C^1$ manifolds endowed with $n$-volume densities and let $\phi:B\hookrightarrow P$ be a proper embedding. Suppose furthermore that $\dim{B}=n$. Let $f:P\ra\mathbb{R}$ be a continuous function with compact support. Then
\[  \int_{\phi(B)}f~dG=\int_{B}J(d\phi_x;F,G)f\circ \phi~dF
\]
\end{prop}
\begin{proof} One considers $\alpha:\Omega\ra B$ to be a chart on $B$ with $\Omega\subset\bR^n$ open bounded and rewrites the two quantities as integrals over $\Omega$. The result is a simple application of Proposition \ref {properties} part (b). 
\end{proof}

Let us give a simple application of Proposition \ref{cov}. We recall the following from \cite{AT}.
\begin{definition} A definition of volume in dimension $n$ associates in a natural (functorial) way to every normed vector space $V$ of dimension $n$ a norm on $\Lambda^nV$  and satisfies a few axioms one of which is the next one. 

\noindent
\;\;(*)\;\; \emph{If $T:(V,\|\cdot\|_V)\ra (W,\|\cdot\|_W)$ is  a linear map between vector spaces of dimension $n$ such that $\|T\|\leq 1$ then $\|\Lambda^nT\|\leq 1$.}
\end{definition}

By a Finsler manifold we will understand a manifold such that the tangent space at each point is endowed with a norm, which varies continuously with the point. Strictly speaking the definition should allow for non-reversible (asymmetric) norms and should include  some strong convexity assumption. 

\begin{prop} Let $\pi:P\ra B$ be a fiber bundle of Finsler manifolds with $B$ compact. Suppose a definition of volume  has been fixed, e.g. the Hausdorff-Busemann. Suppose furthermore that $\|d\pi\|\leq 1$. Then
\[ \vol{(s(B))}\geq \vol(B),
\]
for all sections $s:B\ra P$. 
\end{prop}
\begin{proof} Let $\dim{B}=n$ and let $F$ and $G$ be the $n$-volume densities on $B$ and $P$ that  the definition of volume fixes. Then 
\[\vol{(s(B))}=\int_{s(B)}~dG=\int_BJ(ds;F,G)~dF\geq \int_B~dF=\vol(B),\]
because $J(ds;F,G)=1/J(d\pi\bigr|_{\Imag d\phi};G,F)\geq 1$ since the definition of volume satisfies $(*)$.
\end{proof}
The proposition generalizes thus the next well known result.
\begin{corollary} Let $\pi:P\ra B$ be a Riemannian submersion with $B$ compact. Then 
\[ \vol(s(B))\geq \vol(B),
\]
for all sections $s:B\ra P$.
\end{corollary}
\begin{proof} The differential $d\pi$ is an orthogonal projection.
\end{proof}

\vspace{0.3cm} 

 Let us now review how fiber-integration of densities  works (see, for example Section 7.12 in \cite{GHV}). Suppose $\pi:P\ra B$ is a submersion of $C^1$-manifolds with $\dim{P}=n+m$ and $\dim{B}=m$ and let $\mu$ be a top degree density over $P$. The push-forward  $\pi_*\mu$ of the measure determined by $\mu$ exists always and in this situation is a top degree density on $B$ which is constructed  by first defining the  "retrenchment" $R_{\mu}$ of $\mu$. This is an $n$-density on each fiber $P_b:=\pi^{-1}(b)$ with values in the space of $m$-densities on $T_bB$ which at a point $p\in P_b$ has the expression  
\[ R_{\mu}(p)(w_1\wedge\ldots\wedge w_n)(v_1\wedge\ldots\wedge v_m)=\mu_p(w_1\wedge\ldots\wedge w_n\wedge \tilde{v}_1\wedge\ldots\wedge \tilde{v}_m),\]
for all $w_1\wedge\ldots\wedge w_n\in\Lambda^nV_pP$ where $VP_p:=\Ker d\pi_p$ and $v_1\wedge\ldots\wedge v_m\in \Lambda^nT_bB$ . Here $\tilde{v}_1,\ldots \tilde{v}_m$ are lifts of the $v_i$'s to $T_pP$. The retrenchment does not depend on the lifts $\tilde{v}_1,\ldots \tilde{v}_m$. Then the push-forward density of $\mu$ is defined 
\[ \pi_{*}\mu(b):=\int_{P_b} dR_{\mu}.
\]
Integration on the right hand-side works as follows. The space of (positive, top-dimensional) volume densities on $T_bB$ is in bijection (non-canonically) with the half line $(0,\infty)$, by evaluating the density on a fixed vector $v_1\wedge\ldots\wedge v_m\in\Lambda^mT_bB$.  So we use one such bijection to integrate the density $R_{\mu}$ as if it was a real valued density and interpret the result via the inverse of the same bijection as a density on $T_bB$. 
\begin{remark} Nothing changes if we consider negatively valued densities or even complex ones. 
\end{remark}
The following relation is an application of the standard Fubini theorem, after one chooses trivializing charts on $B$ and $P$:
\begin{equation}\label{Fubini} \int_P\mu=\int_B\pi_*\mu.
\end{equation}

All densities in the next result are assumed symmetric:
\begin{prop}\label{coFub} Let $(P,F,\mu)$ be an $n+m$ dimensional $C^1$ manifold endowed with an $n+m$-volume density $\mu$ and an $n$-density $F$. Let $\pi:P\ra B$ be  a $C^1$ submersion over an $m$-dimensional $C^1$ manifold $B$ also endowed with a top degree volume density $\lambda$. Then
\[ \int_P C(d\pi,\lambda^*,F_{\mu}^*)~d\mu=\int_B\vol_F(\pi^{-1}(b))~d\lambda(b),
\]
where $\vol_F(\pi^{-1}(b))$ is the volume of the fiber with respect to the density $F$. 
\end{prop}
\begin{proof} We use (\ref{Fubini}) for the density $\nu:=C(d\pi,\lambda^*,F_{\mu}^*)~d\mu$. Then the retrenchment of $\nu$ at $p\in P$ is 
\begin{equation}\label{retdens} R_\nu(p)(\omega)(\xi)= F_p(\omega) \lambda_b(\xi), \quad\quad \forall\omega\in\Lambda^mV_pP, \;\forall \xi\in\Lambda^nT_bB,
\end{equation}
which implies that
\[ \int_{P_b}dR_{\nu}=\vol_F(P_b)\lambda_b.
\]
To see that (\ref{retdens}) is true, notice that for every $v_1\wedge\ldots\wedge v_m\in\Lambda^nT_bB$
\[ C(d\pi,\lambda^*,F_{\mu}^*)_p=\frac{F^*_\mu(p)(d_p\pi^*(v_1^*)\wedge\ldots\wedge d_p\pi^*(v_n^*))}{\lambda_b(v_1^*\wedge\ldots\wedge v_m^*)}=\frac{F_p(w_1\wedge\ldots\wedge w_n)\lambda_b(v_1\wedge\ldots\wedge v_m)}{\mu_p(\tilde{v}_1\wedge\ldots\wedge \tilde{v}_m\wedge w_1\wedge\ldots\wedge w_n )},\]
where $w_1,\ldots w_n$ is a basis of $V_pP$ and $\tilde{v}_i$'s are lifts of the $v_i$'s. (see also Example \ref{excodens1}). The last equality holds because 
\[ d\pi^*(v_i^*)(\tilde{v}_j)=\delta^i_j\quad\mbox{and}\quad d\pi^*(v_i^*)(w_j)=0\quad\quad\forall i, j.
\]
\end{proof}
\begin{remark} We can allow for the density $F$ to be non-symmetric but we will require $P$ and $B$ to be oriented manifolds. In this case the  volume of a fiber $\vol_F(\pi^{-1}(b))$ is defined with respect to the induced orientation using the "fiber last" convention. This means that $w_1,\ldots, w_n\in\Ker d_p\pi$ is positively oriented if  and only if $\tilde{v}_1,\ldots, \tilde{v}_m,w_1,\ldots w_n$ is an oriented basis of $T_pP$ for an oriented basis $v_1,\ldots ,v_m$ of $T_{\pi(p)}B$.
\end{remark}

\section{The Lipschitz picture}\label{sec3}

Let us recall the classical Area Formula (Theorem 3.2.3 in \cite{F}).
\begin{theorem}[Area Formula]\label{ArF} Suppose $f:\mathbb{R}^m\ra\mathbb{R}^n$ is a Lipschitz function with $m\leq n$.
\begin{itemize}
\item[(a)] If $A$ is  an Lebesgue measurable set, then
\[\int_A J(d_xf)~d\mathscr{L}(x)=\int_{\mathbb{R}^n} N\left(f\bigr|_{A},y\right)~d\mathcal{H}^m(y),
\]
where $J(d_xf)=\sqrt{\det {(d_xf)^*d_xf}}$, $N(f\bigr|_{A},y)$ is the cardinality of the set $f^{-1}(y)\cap A$ and $\mathscr{L}$ and $\mathcal{H}^m$ represent the Lebesgue respectively the Hausdorff measures.
\item[(b)] If $u$ is an $\mathscr{L}$ integrable function, then
\[ \int_{\mathbb{R}^m} u(x)J(d_xf)~d\mathscr{L}_m=\int_{\mathbb{R}^n}\left(\sum_{x\in f^{-1}(y)}u(x)\right)~d\mathcal{H}^m(y).
\]
\end{itemize}
\end{theorem}
We would like to replace $\mathscr{L}$ and $\mathcal{H}_m$ by two general $m$-volume densities, one on $\mathbb{R}^m$ and one on $\mathbb{R}^n$. In order to do that it is convenient to write the integral on the right as an integral over $f(A)$ which is the support of the function $y\ra N(f\bigr|_{A},y)$. It is worth noting that $f(A)$ is an $\mathcal{H}_m$ measurable set, which is  an implicit statement in the area formula and part of a more general fact.
\begin{prop} Let $X$ and $Y$ be metric spaces such that $X$ is separable, complete. Let $f:X\ra Y$ be a Lipschitz map. Then $f$ takes $\mathcal{H}^m$ measurable sets into $\mathcal{H}^m$ measurable sets.
\end{prop}
\begin{proof} Since $\mathcal{H}^m$ is Borel regular every measurable set $A$ can be written as $A=B\cup N$ with $B$ Borel and $\mathcal{H}^m(N)=0$. Now [2.2.10-13] in \cite{F} explains why the set $f(B)$ is $\mathcal{H}^m$-measurable for every continuous function $f$. On the other hand, Corollary 2.10.11 in \cite{F}, saying that
\begin{equation}\label{LipH} (\Lip f)^m\mathcal{H}^m(C)\geq\int N(f\bigr|_{C},y)~d\mathcal{H}^m(y),\quad\quad \forall C
\end{equation}
implies that $f(N)$ is a set of measure zero.
\end{proof}

Henceforth $f:\bR^m\ra \bR^n$ will denote a Lipschitz map and Rademacher Theorem says that $f$ is differentiable almost everywhere. In the second part of the Area Formula (see also Corollary 3.2.4 in \cite{F})   Federer shows that the set $f(\Crit{f})$ has $m$-Hausdorff measure  zero where 
\[ \Crit(f)=\{x~|~\Ker d_xf\neq 0\}.
\]
So with the exception of an $\mathcal{H}^m$-negligible subset the set $f(A)$ has a natural "tangent" space at every point $f(x)$ and that is $d_xf(\bR^m)$. In fact this tangent space coincides with the approximate tangent space $\mathcal{H}^m$-almost everywhere. For the definition and properties of the almost tangent space of an $(m,\mathcal{H}^m)$ countably rectifiable set see \cite{Si} (Def. 11.2 and Remark 11.7). It is important to be able to identify the tangent space of $f(A)$ in this manner since we want to define an orientation of $f(A)$, meaning the selection of a "positive" side to $\Lambda^mT_{p}f(A)$. By definition the positive orientation vector on $f(A)$ will be $\Lambda^m d_xf(e_1\wedge\ldots\wedge e_m)$ at all points where $d_xf(\bR^m)$ coincides with the approximate tangent space, which means almost everywhere. 

 We can define for every $m$-density $G$ on $\mathbb{R}^n$
\[ \int_{f(A)}~dG(y):=\int_{f(A)} G(y,\xi_y)~d\mathcal{H}^m(y),
\]
where $\xi_y\in\Lambda^m T_yf(A)$ is one of the two vectors of euclidean norm $1$ if $G$ is symmetric or is the unique positive oriented vector of norm $1$ if $G$ is not symmetric.  This is of course the definition of the integral of an $m$-parametric integrand over an $m$-dimensional integer-multiplicity current (see page 515 in \cite{F}) which for $f(A)$ a $C^1$-manifold coincides with the one given in Section \ref{Sec2}. 

One can extend  this definition to obtain a measure  on $f(A)$ denoted $G$, which is absolutely continuous with respect to $\mathcal{H}^m$ and such that the Radon-Nykodim derivative of $G$ by $\mathcal{H}^m$ is a jacobian:
\[ \frac{dG}{d\mathcal{H}^m}(y)=J(\id_{T_yf(A)};\mathcal{H}^m,G)=G(y,\xi_y)
\] 
Above, the notation $\mathcal{H}^m$ plays also the role of the euclidean $m$-density on $\mathbb{R}^n$. A measurable function $v$ is called $G$-integrable if $\int_{f(A)}|v|~dG<\infty$.

\begin{theorem} Suppose $f:\mathbb{R}^m\ra\mathbb{R}^n$ is a Lipschitz function with $m\leq n$ and let $F$ be an $m$-density on $\mathbb{R}^m$ and $G$ be an $m$-density on $\mathbb{R}^n$.
\begin{itemize}
\item[(a)] If $A$ is  a Lebesgue measurable set, then
\[\int_A J(d_xf;F,G)~dF(x)=\int_{f(A)} N\left(f\bigr|_{A},y\right)~dG(y),
\]
\item[(b)] If  $u$ is a non-negative measurable function or $u$ is  an $F$-integrable function  over $A$, then
\[ \int_{A} u(x)J(d_xf;F,G)~dF=\int_{f(A)}\left(\sum_{x\in f^{-1}(y)}u(x)\right)~dG(y).
\]
\end{itemize}
\end{theorem}
\begin{proof}  $(a)$ Notice first that by (\ref{LipH})   both sides vanish on the set of points where $df$ does not have maximal rank.  It is therefore enough to consider
 \[A\subset \{x~|~df_x \mbox{ exists and has maximal rank}\}.\] 
In this situation for $y=f(x)$ we have $T_yf(A)=d_xf(\mathbb{R}^n)$, (at the points where the tangent space exists). The next relation together with  part $(b)$ of Theorem \ref{ArF} finishes the proof.
\[ J(\id_{T_x\mathbb{R}^m}; \mlH^m,F)\cdot J(d_xf; F,G) =J(d_xf;\mlH^m,\mlH^m)\cdot J({\id_{\Imag{d_xf}}}; \mlH^m,G)
\]
Due to Proposition \ref{properties}, part (b), both sides equal $J(d_xf; \mlH^m, G)$.

$(b)$ For $u$ non-negative, write $u=\sum_{i=1}^{\infty}\frac{1}{i}\chi_{A_i}$ for appropriately chosen measurable sets $A_i$ (see Theorem 7, Sect. 1.1.2 in \cite{EG}) and use the Monotone Convergence Theorem.
\end{proof}
The situation is not much different for the coarea formula (Theorem 3.2.11 in \cite{F}). However, this time we have to deal with $3$ densities. 
We state it first and then explain.
\begin{theorem}\label{coarea} Suppose $f:\mathbb{R}^{n+m}\ra\mathbb{R}^n$ is a Lipschitz function  and let $F$, $\mu$ be volume densities on $\mathbb{R}^{n+m}$ of degree $m$ and $m+n$ respectively. Let $\lambda$ be an $n$ volume density on $\mathbb{R}^n$ and let $A\subset\mathbb{R}^{n+m}$ be a Lebesgue measurable set.  Then
\begin{itemize}
\item[(a)] \[\int_A C(d_xf;\lambda^*,F_{\mu}^*)~d\mu(x)=\int_{\mathbb{R}^n}\vol_F(A\cap f^{-1}(y))~d\lambda(y).\]
\item[(b)] If $g:\mathbb{R}^{n+m}\ra\mathbb{R}^n$ is either a non-negative measurable function or $g$ is $\mu$-integrable then
\[ \int_A g(x)\cdot C(d_xf;\lambda^*,F_{\mu}^*)~d\mu(x)=\int_{\mathbb{R}^n}\left(\int_{A\cap f^{-1}(y)}g(x)~dF\right)~d\lambda(y).
\]
\end{itemize}
\end{theorem}
 It is essential for the good definition of the right hand side  that $A\cap f^{-1}(y)$ is $\mlH^m$- countably rectifiable for $\mathscr{L}^n$ almost all $y\in\mathbb{R}^n$ (this is  Theorem 3.2.15 in \cite{F}) and that $y\ra \mlH^m(A\cap f^{-1}(y))$ is measurable (check Lemma 1, Sect. 3.4 in \cite{EG}). There is one other issue: how to define an orientation on $A\cap \pi^{-1}(y)$ in the case when $F$ is non-symmetric. We use the weak form of Sard's Theorem which is a consequence of the Coarea formula (Theorem 3.2.11 in \cite {F}) and states that
 \[ \mathcal{H}^m(\Crit{f}\cap f^{-1}(y))=0\quad\quad \mathcal{L}^n-a.e. \;y\in \bR^n.
 \]
Here $\Crit{f}$ represent all those points where $d_xf$ is not surjective. So with the exception of an $\mlH^m$-negligible set every point $p\in f^{-1}(y)$ has a tangent space  $\Ker d_xf$. One gives  to this space the orientation using the map $d_xf$ and the canonical orientations of $\bR^n$ and $\bR^{n+m}$ (fiber last convention). 
\begin{proof} $(a)$
We are allowed again to make the simplifying assumption that the differential of $f$ has maximal rank at all points of $A$. This is because both sides vanish on the rest of the points since the integrals are zero with respect to the Hausdorff measure as in the proof of the standard Coarea formula. 

We claim now that
\begin{equation*} C(d_xf;\lambda^*,F_{\mu}^*)\cdot J(\id_{T_x\mathbb{R}^{n+m}}; \mathscr{L},\mu)\stackrel{(\star)}{=}C(d_xf; \mathscr{L}^*,\left(\mlH^m\right)^*_{\mathscr{L}})\cdot J(\id_{\Ker{d_xf}}; \mlH^m,F)\cdot J(\id_{\Imag{ d_xf}};\mathscr{L},\lambda).
\end{equation*}
This happens on one hand because part $(d)$ of Proposition \ref{properties} implies that 
\[ C(d_xf; \mathscr{L}^*,\mlH^*_{\mathscr{L}})\cdot J(\id_{\Imag{ d_xf}};\mathscr{L},\lambda)=C(d_xf;\lambda^*,\left(\mlH^m\right)^*_{\mathscr{L}}),
\]
while part $(e)$ of the same proposition implies that
\[C(d_xf;\lambda^*,\left(\mlH^m\right)^*_{\mathscr{L}})\cdot J(\id_{\Ker{d_xf}}; \mlH^m,F)=C(d_xf;\lambda^*,F_{\mu}^*)\cdot J(\id_{T_x\mathbb{R}^{n+m}}; \mathscr{L},\mu).
\]
One can use the standard Coarea formula to finish the proof since
\[ \int_A C(d_xf;\lambda^*,F_{\mu}^*)~d\mu(x)=\int_AC(d_xf;\lambda^*,F_{\mu}^*)\cdot J(\id_{T_x\mathbb{R}^{n+m}}; \mathscr{L},\mu)~d\mathscr{L}=\]\[=\int_AC(d_xf; \mathscr{L}^*,\left(\mlH^m\right)^*_{\mathscr{L}})\cdot J(\id_{\Ker{d_xf}}; \mlH^m,F)\cdot J(\id_{\Imag{ d_xf}};\mathscr{L},\lambda)~d\mathscr{L}=\]\[=\int_{\mathbb{R}^n}\left(\int_{A\cap f^{-1}(y)} J(\id_{\Ker{d_xf}}; \mlH^m,F)~d\mlH^m\right)\cdot J(\id_{\Imag{ d_xf}};\mathscr{L},\lambda)~d\mathscr{L}(y)=\]\[=
\int_{\mathbb{R}^n}\left(\int_{A\cap f^{-1}(y)}~dF\right)~d\lambda.
\]

$(b)$ The same ideas as for part $(b)$ in the Area Formula work here.
\end{proof}
\begin{remark} The Theorem \ref{coarea} was stated for symmetric densities. We can allow for $F$ to be non-symmetric in which case $\vol_F$ is defined using the orientation, as in the comments preceding the proof.
\end{remark}
We can improve now Proposition \ref{coFub} by removing the submersion condition. Notice that on a general smooth manifold $P$ it makes sense to talk about measurable sets. By this we understand that the set is measurable in every chart of a fixed atlas with respect to the Lebesgue  measure induced by the chart. Clearly this does not depend on the particular atlas used. Moreover this is equivalent to saying that the set is measurable with respect to any top degree volume density on the manifold. 
\begin{corollary}\label{funcman} Let $\pi:P\ra B$ be a $C^1$ map between smooth manifolds with $\dim{P}=m+n$ and $\dim{B}=n$. Let $F$ and $\mu$ be volume densities of degrees $m$ and $n+m$ on $P$ and $\lambda$ be an $n$-density on $B$. Let $A$ be a measurable set. Then
\begin{itemize}
\item[(a)] \[\int_A C(d_x\pi;\lambda^*,F_{\mu}^*)~d\mu(x)=\int_{B}\vol_F(A\cap \pi^{-1}(y))~d\lambda(y).
\]
\item[(b)] If $g:P\ra B$ is a measurable function such that $g$ is $\mu$ integrable then
\[ \int_A g(x)\cdot C(d_x\pi;\lambda^*,F_{\mu}^*)~d\mu(x)=\int_{B}\left(\int_{A\cap \pi^{-1}(y)}g(x)~dF\right)~d\lambda(y).
\]
\end{itemize}
\end{corollary}
\begin{proof} By choosing a partition of unity $(D_i,\tau_i)$ on $B$ and using the result for $\pi\bigr|_P$ where $P=\pi^{-1}(D_i)$ and $g=\tau_i$ one is reduced to prove the claim for $B=\mathbb{R}^n$.

Choose a covering of $P$ with open sets $U_i$ diffeomorphic with $\mathbb{R}^{n+m}$ and a partition of unity $\rho_i$ subordinate to $U_i$. Since $\pi$ is $C^1$, by  shrinking the $U_i$'s if necessary, we can assume that $\pi\bigr|_{U_i}$ is Lipschitz (as a map $\mathbb{R}^{n+m}\ra\mathbb{R}^n$).  One then uses Theorem \ref{coarea} for $\pi\bigr|_{U_i}$ and $h=g\rho_i$ and sums up.
\end{proof}

\begin{remark}\label{funcmanrem} An important particular case of Corollary \ref{funcman} occurs when $\pi=f:P\ra \mathbb{R}$, $\lambda=dt$, $\mu=|\Omega|$ a volume form and $F$ arises by "contracting" the volume form $\Omega$   with a convex $1$-density $\bar{F}$, meaning that $F=\bar{F}^*\circ \iota_{\Omega}^*$ where $\bar{F}^*$ is the functional dual of $\bar{F}$.  In this case, the cojacobian is $\bar{F}^*(df)=\bar{F}(\nabla f)$ (see Example \ref{ex2}) and one recovers Shen's Theorem 3.3.1 in \cite{Shen}.  
\end{remark}

\section{The anisotropic outer Minkowski content and the Sobolev inequality}\label{sec4}
The codensities introduced in Section \ref{sec1} are,  at least  in degree $1$, important objects  in anisotropic geometry (see \cite{An}, \cite{P} and \cite{PK} and the references therein  for more details). We review the main concepts following along the lines of Andrews presentation in \cite{An}.    The starting point is an \emph{oriented} $(n+1)$-dimensional vector space $V$, endowed with  a volume form $\Omega\in\Lambda^{n+1}V^*$ compatible with the orientation of $V$ and a degree $n$ volume density $F$, which in general is not assumed to be origin symmetric but is assumed to be extendibly convex, i.e. it is the restriction of a convex, positively homogeneous function $F:\Lambda^nV\ra\mathbb{R}$. 

 We have several objects associated to $F$ and $\Omega$. On one hand, $F$ and $\Omega$ induce a Minkowski functional $\bar{F}^*$ on $V^*$ (see Proposition \ref{altcodens}) as follows:
\[ \bar{F}^*:=F\circ \left(\iota_{\Omega}^*\right)^{-1}.
\]
One relates this to what was done in Section \ref{sec1} as follows. Start with a parametric integrand $F$  of degree $n$ and a volume form $\Omega$. We can easily modify the Definition-Lemma \ref{codens}, asking for the basis $\{v_1\ldots v_m,w_1,\ldots w_n\}$ to be positively oriented in order to obtain a well-defined  parametric integrand  $F_{\Omega}^*$ on $\Lambda^mV^*$  which coincides with $\bar{F}$ when $m=1$ since  the proof of Proposition \ref{altcodens} is the same. 
 
We will stay with the case $m=1$. The set $W_F:=\{v\in V~|~\bar{F}\leq 1\}$ is called the Wulff shape, where 
\[ \bar{F}(v)=\sup_{\bar{F}^*(f)\leq 1} f(v).
\] 
In other words, $\bar{F}^*$ is the support function of $W_F$ (see \cite{Sc}). To be more precise, suppose $V$ is endowed with an inner product  $\langle\cdot,\cdot\rangle$  such that $\Omega$ is the volume form induced by $\langle\cdot,\cdot\rangle$.  Let $h_{W_F}$ be the support function of $W_F$ on $V$. Then
\begin{equation}\label{Fh} h_{W_F}(u)=\bar{F}^*(\langle u,\cdot \rangle)
\end{equation}

Let us mention next a multiplication property of the densities involved. Suppose that $v_1\ldots v_n$ are $n$ linearly independent vectors and $w$ is a vector such that the hyperplane spanned by $\{v_1,\ldots,v_n\}$ is supporting for $\lambda W_F$ where $\lambda>0$ is such that $w\in \partial (\lambda W_F)$. Such a vector is said to be anisotropically orthogonal to the hyperplane spanned by $v_1,\ldots, v_n$.   Then
\begin{equation}\label{multprop} |\Omega|(w\wedge v_1\wedge\ldots\wedge v_n)= \bar{F}(w)\cdot F(v_1\wedge\ldots\wedge v_n).
\end{equation}
Indeed we can reinterpret this equality as
\begin{equation}\label{FsFb} \bar{F}(w) \bar{F}^*(w^*)=1,
\end{equation}
where $w^*$ is the dual of $w$ with respect to the basis $w, v_1,\ldots , v_n$. Let $\tilde{w}:=\frac{1}{\lambda} w\in W_F$. We will prove that
\begin{equation}\label{bFw} \bar{F}^*(w^*)=w^*(\tilde {w}),
\end{equation}
from  which (\ref {FsFb}) follows since $\bar{F}(w)=\lambda$, by definition. We show in fact that $\tilde{w}$ realizes the norm of $w^*$, which equals $F^*(w^*)$.

The fact that $H:=\langle v_1,\ldots v_n\rangle$ is supporting for $W_F$ at $\tilde{w}$ can be rephrased as saying that
\[\bar{F}(\tilde {w}+v)\geq \bar{F}(\tilde{w}),\quad\quad \forall v\in H. 
\]
Every $z\in W_F$ can be written as $z=a\tilde{w}+v$, with $v\in H$. Since we are looking for those $z$ that realize $\sup_{z\in W_F} w^*(z)$, we can assume that $a>0$. From
\[ 1\geq \bar{F}(z)=a \bar{F}(\tilde{w}+\frac{1}{a}v)\geq a\bar{F}(\tilde{w})=a,
\]
we conclude that $w^*(z)\leq w^*(\tilde{w})$ for all $z\in W_F$, hence (\ref{bFw}).
\begin{remark} If $h_{W_F}$ happens to be smooth everywhere then for each hyperplane $H$ there are exactly two anisotropic normal directions pointing to different half-spaces. One can single out a unit anisotropic normal $n$ for an oriented $H$ by requiring that $n\in\partial W_F$ and $n\wedge v_1\wedge\ldots \wedge v_n$ be positively oriented in $V$ for $v_1\wedge\ldots\wedge v_n$ positively oriented in $H$.
\end{remark}

\begin{lemma} Let $U\subset V$ be an oriented subspace of dimension $n$. Then 
\[ J(\id_{U};\mathcal{H}^n\bigr|_{U},F\bigr|_{U}) = h_{W_{F}}(\nu_U),
\]
where $\nu_U$ is the unit  normal to $U$ with respect to the inner product, such that $\nu, u_1,\ldots, u_n$ is a positively oriented basis of $V$  for every oriented basis $u_1,\ldots, u_n$ of $U$. In particular if $\Sigma$ is an oriented $C^1$ hypersurface in $V$, then
\[\int_\Sigma~dF=\int_{\Sigma}h_{W_{F}}(\nu_\Sigma)~d\mathcal{H}^n.
\]
\end{lemma}
\begin{proof} For $\{v_1,\ldots, v_n\}$ an oriented basis of $U$ one has
\[J(\id_{U};\mathcal{H}^n\bigr|_{U},F\bigr|_{U})=\frac{F(v_1\wedge\ldots\wedge v_n)}{\mathcal{H}^n(v_1\wedge\ldots\wedge v_n)}=\frac{F(v_1\wedge\ldots\wedge v_n)}{\Omega(\nu_U\wedge v_1\wedge\ldots\wedge v_n)}=\bar{F}^*(\nu_U^*)=\]\[ =\bar{F}^*(\langle\nu_U,\cdot\rangle)=h_{W_{F}}(\nu_U).\]
We used the fact that $\nu_U^*$ which is the element of the dual basis to $\{\nu_U,v_1,\ldots,v_n\}$ that returns $1$ when evaluated on $\nu_U$ is in fact equal to $\langle \nu_U,\cdot\rangle$ independently of the choice of $v_1,\ldots, v_n$.
\end{proof}
\begin{remark} If $\Sigma$ is the smooth boundary of an open set in $V$ then we use the "outer normal first" convention to orient $\Sigma$ and in this case $\nu$ of the previous lemma is the unit outer normal.
\end{remark}
Let $B$ be a bounded subset of $V$. All the volumes of the sets that appear below are computed with respect to $\Omega$. The anisotropic outer Minkowski content is the quantity
\[ \SM_F(B)=\lim_{t\searrow 0}\frac{\vol(B+tW_F)-\vol(B)}{t},
\]
when the limit exists. 
\begin{theorem}\label{anisoMink} Let $B$ be a bounded domain with $C^2$ boundary. Then
\begin{equation}\label{oMc}\SM_F(B)=\int_{\partial B}~dF.
\end{equation}
\end{theorem}
\begin{proof} We will assume first  that $W_F$ has support function of class $C^{\infty}$ (although $\partial W_F$ might still have singularities). This implies in particular that $W_F$ is strictly convex (see Corollary 1.7.3 in \cite{Sc}).

Consider the $\bar{F}$-distance from $\Sigma:=\partial B$ to a point $p$ outside $B$. This is by definition the infimum over the set of all piecewise smooth curves $\gamma:[0,1]\ra V$ that connect a point on $\Sigma$ with $p$ of the following integral
\[\int_{\gamma}~d\bar{F}=\int_0^1\bar{F}(\gamma'(t))~dt.
\]
Due to the fact that $\bar{F}$ is convex (and constant) the geodesics of this action functional are straight lines.  Consider the anisotropic Gauss map $n:\Sigma\ra \partial W_F$ which associates to a point $b$ the unique vector $n(b)\in \partial W_F$ such that $T_b\Sigma$ is a supporting hyperplane for $W_F$ at $n(b)$ and $n(b)\wedge \orient{T_b\Sigma}=\orient V$. With a choice of an inner product as above, this Gauss map can be seen as the composition of the standard Gauss map  with the gradient of the support function $h_{W_F}$.  Clearly $n$ is of class $C^1$ and due to the compactness of $\Sigma$ the map
\[ \phi:[0,t]\times\Sigma\ra V, \quad\quad (b,s)\ra b+sn(b)
\]
is injective for $t$ sufficiently small. This would be true for every non-vanishing vector field $n$ along $\Sigma$. The points in the image of $\phi$ are those such that the $\bar{F}$-distance from $\Sigma$ to them is at most $t$.  Moreover, for such a $t$ we have
\[\tilde{B}_t:=\overline{B}+tW_F\setminus \overline{B} =\Imag{\phi}.\]
Indeed the $\supset$ inclusion is trivial while $\subset$ uses the fact that a point $x+sy\in B+tW_F\setminus B$ with $x\in B$, $y\in W_F$ and $s\leq t$ is at an $\bar{F}$ distance at most $t$ from $\Sigma$.

Therefore $\phi$ is an oriented $C^1$ diffeomorphism onto $\tilde{B}_t$.  Using the density $\lambda\times F$ on $[0,t]\times \Sigma$ where $\lambda$ is the Lebesgue measure on the line and $\Omega$ on $\tilde{B}_t$, one gets:
\[ \vol(\tilde{B}_t)=\int_{[0,t]\times\Sigma} J(d\phi;\lambda\times F, \Omega)~d\lambda\otimes dF,
\]
If we can prove that for $t$ sufficiently small $|J(d\phi;\lambda\times F, \Omega)_{(s,b)}-1|\leq Mt$ for all $s\leq t$ for some $M>0$ then  
\[ \frac{1}{t}\left|\int_{[0,t]\times\Sigma}1-J(d\phi;\lambda\times F, \Omega)~d\lambda\otimes dF\right|\leq \int_{[0,t]\times\Sigma}M~d\lambda\otimes dF \stackrel{t\searrow 0}{\ra}0.
\]
and this would conclude the proof in this case. Now
 \begin{equation}\label{jaco} J(d\phi;\lambda\times F, \Omega)_{(s,b)}=\frac{|\Omega|(n(b)\wedge d_b\phi_s(v_1)\wedge\ldots \wedge d_b\phi_s(v_n))}{F(v_1\wedge\ldots\wedge v_n)},\end{equation}
 for every oriented basis $\{v_1,\ldots,v_n\}$ of $T_b\Sigma$. Observe that by the multiplication property (\ref{multprop}) for $s=0$, $J(d\phi;\lambda\times F, \Omega)=1$ and in fact the numerator is a polynomial in the variable $s$ and therefore $|J(d\phi;\lambda\times F, \Omega)_{(s,b)}-1|\leq Ms$ for $s$ small. 

In order to prove (\ref{oMc}) for a general convex body $W_F$ which contains the origin in the interior we use the fact that $W_F$ can be approximated with respect to the Hausdorff metric   from the inside and from the outside with convex bodies of the type used in the first part of the proof. This follows from Theorem 1.8.13,  Lemma 1.8.4 and Theorem 3.3.1 of \cite{Sc}. This means in particular that there exist  two  sequences  of support functions   that uniformly (over every compact subset of $V$) approximate $h_{W_F}$ from below and from above.  Let $K_2\subset W_F\subset K_1$ be two convex sets as above such that 
\begin{equation*}\sup_{\Sigma} |h_{K_1}-h_{K_2}|<\frac{\epsilon}{6\mathcal{H}^n(\Sigma)},
\end{equation*}
This gives 
\begin{equation} \label{ineq1}\left|\int_{\Sigma}h_{K_1}-h_{K_2}\right|<\frac{\epsilon}{6}\end{equation}
 and a similar inequality with $h_{W_F}$ instead of $h_{K_1}$. Now choose $\delta >0$ such that 
\begin{equation}\label{ineq2}\left|\frac{\vol(B+tK_i\setminus B)}{t}-\int_{\Sigma}h_{K_i}~d\mathcal{H}^n\right|<\frac{\epsilon}{12}, \quad\quad\forall 0<t<\delta\quad \forall  i\in\{1,2\} .
\end{equation}
Combining (\ref{ineq1}) and (\ref{ineq2}) we get that
\[  \left| \frac{V(B+tK_1\setminus B)-V(B+tK_2\setminus B)}{t}\right|<\frac{\epsilon}{3}.
\]
We put everything together
\[ \left|\frac{V(B+tW_F\setminus B)}{t}-\int_{\Sigma}h_{W_F}~d\mathcal{H}^n\right|\leq \left| \frac{V(B+tK_1\setminus B)-V(B+tK_2\setminus B)}{t}\right|+\]\[+\left|\frac{\vol(B+tK_2\setminus B)}{t}-\int_{\Sigma}h_{K_2}~d\mathcal{H}^n\right|+ \left|\int_{\Sigma} h_{K_2}-h_{W_F}~d\mathcal{H}^n\right|<\frac{\epsilon}{3}+\frac{\epsilon}{12}+\frac{\epsilon}{6}<\epsilon
\]
\end{proof}
Recall now the Brunn-Minkowski inequality (Theorem 3.2.41 in \cite{F}). Let $A,B\subset V$ be two non-empty sets. Then
\[ \vol(A+B)^{\frac{1}{n+1}}\geq \vol(A)^{\frac{1}{n+1}}+\vol(B)^{\frac{1}{n+1}}
\]
Taking $A=tW_F$ and $B$ a domain with $C^2$ boundary we get:
\[ \frac{\vol(B+tW_F)^{\frac{1}{n+1}}-\vol(B)^{\frac{1}{n+1}}}{t}\geq \vol{(W_F)}^{\frac{1}{n+1}}
\]
and letting $t\ra 0$ we arrive at the \emph{anisotropic isoperimetric inequality}:
\[ \frac{1}{n+1}\vol(B)^{-\frac{n}{n+1}}\int_{\partial B}~dF\geq \vol(W_F)^{\frac{1}{n+1}}.
\]
This can be rewritten in the more familiar form
\begin{equation}\label{isoper} \Area_F(\partial B)\geq (n+1)\vol{(W_F)}^{\frac{1}{n+1}}\vol{(B)}^{1-\frac{1}{n+1}}.
\end{equation}

 We emphasize that the anisotropic isoperimetric inequality compares the \emph {exterior} anisotropic area of a closed surface with a certain power of the volume of the enclosed region. If $W_F$ is not origin symmetric then the interior anisotropic area can be different.

The anisotropic Sobolev inequality follows from the anisotropic isoperimetric inequality using the same trick as in \cite{FF}. To wit, let $f$ be a smooth function with compact support and let 
\[ A_t:=|f|^{-1}(t,\infty),\quad\quad\quad B_t:=|f|^{-1}\{t\}
\]
By Sard's theorem $B_t$ is a smooth hypersurface that bounds the domain $A_t$ for almost all $t>0$.  At a regular point $x\in\mathbb{R}^{n+1}$, let $\tilde{\partial}_t$ be a lift of the canonical vector $1\in\mathbb{R}$, i.e. $d_x|f|(\tilde{\partial}_t)=1$.    Let  $v_1,\ldots, v_n$ be an oriented basis  of $T_xB_t$, orienting $B_t$ as a boundary of $A_t$. The cojacobian (see (\ref{coarea})) of $|f|:\mathbb{R}^{n+1}\ra \mathbb{R}$ is then:
\[ C(d_x|f|;\lambda^*,\bar{F}^*)=\frac{F(v_1\wedge\ldots\wedge v_n)}{|\Omega|(\tilde{\partial}_t\wedge v_1\wedge\ldots\wedge v_n)}=\frac{F(v_1\wedge\ldots\wedge v_n)}{\Omega(-\tilde{\partial}_t\wedge v_1\wedge\ldots\wedge v_n)}=\]\[=\bar{F}^*(-d_x|f|)\stackrel {(\ref{Fh})}{=}h_{W_F}(-\grad_x |f|).
\]
The reason for the appearance of the minus sign is that $\tilde{\partial}_t\wedge v_1\wedge\ldots\wedge v_n$ is negatively oriented as $\tilde{\partial}_t$ points towards the interior of $A_t$ and, as $t$ increases, $A_t$ decreases. We therefore have:
\[ \int_{V} h_{W_F}(-\grad_x |f|)~d\Omega=\int_0^\infty \Area_F(|f|^{-1}\{t\})~dt.
\]
It follows from (\ref{isoper}) that
\begin{equation}\label{ineq3} \int_{V} h_{W_F}(-\grad_x |f|)~d\Omega\geq (n+1)\vol{(W_F)}^{\frac{1}{n+1}}\int_0^{\infty} \vol(A_t)^{1-\frac{1}{n+1}}~dt.
\end{equation}
We quickly review a  result, proved at page 487-488 in \cite{FF}.
\begin{lemma} \[\int_0^\infty\vol(A_t)^{1-\frac{1}{n+1}}~dt\geq \left(\displaystyle\int_{V}|f|^{\frac{n+1}{n}}~d\Omega\right)^{\frac{n}{n+1}}.\]
\end{lemma}
\begin{proof} Let $f_t$ be the truncation of $f$ at levels $t$ and $-t$. The function $u_t:=\left(\displaystyle\int_V |f_t|^{\frac{n+1}{n}}\right)^{\frac{n}{n+1}}$ is increasing and Lipschitz continuous. Indeed it follows from  $|f_{t+h}|\leq |f_t|+h\chi_{A_t}\; (h>0)$ and Minkowski inequality that $u(t+h)-u(t)\leq h \vol(A_t)^{\frac{n}{n+1}}\leq h \cdot C$. From this last inequality it also follows that
\[\left(\int_V |f|^{\frac{n+1}{n}}\right)^{\frac{n}{n+1}}=\int_0^{\infty} u'(t)~dt\leq \int_0^{\infty} \vol(A_t)^{\frac{n}{n+1}}~dt.
\]
\end{proof}
 Putting together the previous Lemma and (\ref{ineq3}) one gets \emph{the anisotropic Sobolev inequality}
\begin{equation} \label {Soboin}\int_{V} h_{W_F}(-\grad_x |f|)~d\Omega\geq (n+1)\vol{(W_F)}^{\frac{1}{n+1}}\|f\|_{L^{(n+1)'}(\mathbb{R}^n)}.
\end{equation}
where $(n+1)'=\frac{n+1}{n}$. This inequality implies of course  the anisotropic Sobolev embedding
\[ W^{1,1}_F(V)\hookrightarrow L^{(n+1)'}(V,\Omega),
\]
where $W^{1,1}_F(V)$ is the completion of the normed vector space $C^{\infty}_c(V)$ with the norm:
\[ \|f\|_{W^{1,1}_F}:=\int_{V} \bar{F}^*(-d|f|)~d\Omega +\int_V |f|~d\Omega.
\]
Applying (\ref{Soboin}) to $f^v$ with $v=\frac{pn}{n+1-p}$, $p>1$ and using Holder's inequality  just like in \cite{FF}, page 488 one gets the anisotropic Gagliardo-Nirenberg inequality:
\[ \left(\int |f|^{p^*}\right)^{\frac{1}{p^*}}\leq C \left(\int |h_{W_F}(-\grad |f|)|^p\right)^{\frac{1}{p}},
\]
where $\displaystyle{\frac{1}{p^*}=\frac{1}{p}-\frac{1}{n+1}}$ and $C=\displaystyle{ \frac{pn}{(n+1-p)(n+1)}}\vol{(W_F)}^{-\frac{1}{n+1}}$. This induces the embedding $W^{1,p}_F\hookrightarrow L^{p^*}$.
\vspace{0.5cm}

\begin{remark} There exists a vast literature in which classical problems in Riemannian Geometry are transferred to the anisotropic world. For example, with the above coarea formula at hand and the anisotropic isoperimetric inequality one can give a very short proof of the P\'olya-Szeg\"o principle (Faber-Krahn inequality) in the anisotropic world in which one uses anisotropic symmetrization instead of the usual Steiner symmetrization for functions. All this is treated in great generality by Van Schaftingen in \cite{vS} with a precursor in \cite{AFLT}. One has to mention also the recent quantitative anisotropic isoperimetric  inequality of Figalli, Maggi and Pratelli \cite{FMP}. 
\end{remark}

\section{A tube formula and an anisotropic connection}\label{sec5}
 The proof of Theorem \ref{anisoMink} gives us "for free" an  anisotropic "half-tube" formula for hypersurfaces. We will restrict attention to those Wulff shapes $W_F$  for which the support function $h_{W_F}$ is smooth which implies the differentiability of the anisotropic  Gauss map $n:\Sigma\ra \partial W_F$. The hypersurfaces  $\Sigma$ considered below are $C^2$ and without boundary.
 
   \begin{definition}\label{secff} The (positively oriented) anisotropic shape operator of an oriented hypersurface $\Sigma$ is the bundle   endomorphism $S^{F}:T\Sigma\ra T\Sigma$, which at a point $b$ is defined by
 \[ S^{F}_b(v)=P_b(d_bn(v)),
 \]
 where $n:\Sigma\ra \partial W_F$ is the anisotropic Gauss map corresponding to the orientation and $P_b$ is the projection to $T_b\Sigma$ along the linear subspace determined by $n(b)$. The fundamental symmetric polynomials in the eigenvalues of $S^{\bar{F}}$ are denoted by $c_k(S^{F})$.
 \end{definition}  
 
 With this definition we  see that (\ref{jaco}) equals $|\det(1+sS^{F})|$ and for $s$ small one can forget about $|\cdot|$.  We therefore get
 \begin{theorem}  Let $\Sigma\subset V$ be a smooth hypersurface oriented using the normal that points to the unbounded component of $V\setminus \Sigma$. Let $\vol T^+_{\Sigma}({\varepsilon})$ be the volume of the set of points that lie in the unbounded component at an $\bar{F}$-distance at most $\varepsilon$ from $\Sigma$. Then for $\varepsilon$ sufficiently small the following holds:
 \[ \vol T^+_{\Sigma}({\varepsilon})=\sum_{k=0}^{n}\frac{\varepsilon^{k+1}}{k+1}\int_{\Sigma}c_k(S^{F})~dF.
 \] 
 \end{theorem}
 \begin{corollary} 
 \[ \Area_F(\partial W_F)=(n+1)\vol(W_F).
 \]
  \end{corollary}
\begin{proof}
 Take $\Sigma=W_F$ oriented via the exterior anisotropic normal. Then $S_b=\id_{T_{b}\Sigma}$ and
 \[((1+\varepsilon)^{n+1}-1) \vol(W_F)= \vol T^+_{W_F}({\varepsilon})=\Area_F(\partial W_F)\int_0^{\epsilon}(1+s)^n.
 \]
 \end{proof}
  To get the full tube formula we need to consider the negatively oriented anisotropic Gauss map $n^-:\Sigma\ra \partial W_F$ which associates to every $b\in\Sigma$ the unique vector $n^-(b)\in \partial W_F$ such that $T_{n^-(b)}\partial W_F=T_b\Sigma$ and $n^-(b)\wedge \orient \Sigma=-\orient V$. The negatively oriented shape operator $S^{F}_-$ is defined analogously to  \ref{secff}.  Notice that when $\bar{F}$ is reversible, i.e. $\bar{F}(v)=\bar{F}(-v)$, one has the relation
 \[ S^{F}_-=-S^{F}.
 \]

 \begin{corollary} Let $\Sigma\subset V$ be a smooth hypersurface as above and let $T_{\varepsilon}=\Sigma +\varepsilon W_F$. Then for  $\varepsilon$ sufficiently small the following holds:
 \[ \vol(T_{\epsilon})=\sum_{k=0}^{n}\frac{\varepsilon^{k+1}}{k+1}\int_{\Sigma}c_k(S^{F})+c_k(S^{F}_-)~dF.
 \]
  \end{corollary}

  Inspired by this formula we introduce

  \begin{definition} The anisotropic connection is the differential operator $\Gamma(T\Sigma)\times\Gamma(V)\ra \Gamma(T\Sigma):$
  \[ \nabla^{F}_{Y}(X)=P^{n}_{T\Sigma}(dX(Y)).
  \]
  where $P^n_{T\Sigma}$ is the projection along $\langle n\rangle$ onto $T\Sigma$. This is a true connection on $\Sigma$ when restricted to tangent vector fields $X:\Sigma\ra T\Sigma$.
  
  The anisotropic divergence along the hypersurface $\Sigma$ of a vector field $X$  is:
  \[ \diver^{F}_\Sigma X:=\tr \{Y\ra \nabla^F_YX\}.
  \]
  \end{definition}
  
  We will consider the functional:
  \[       \mathscr{F}(\Sigma):= \int_\Sigma~dF
  \]
  and look at its first variation
 \begin{theorem}\label{firstvar} Let $X$ be a vector field on $V$. Then 
 \[ \frac{\partial \mathscr{F}}{\partial X}(\Sigma)=\int_{\Sigma} \diver^F_{\Sigma}(X)~dF.
 \] 
 \end{theorem}
\begin{proof} We have to compute:
\[ \frac{d}{dt}\biggr|_{t=0} \int_{\Sigma_t} ~dF=\int_{\Sigma} d\lim_{t \ra 0}\frac{\phi^*_t(F)-F}{t}.
\]
 where $\phi:(-\epsilon,\epsilon)\times V\ra V$ is the flow generated by $X$. We will let $v_1,\ldots, v_{n-1}$ be a basis of $T_p\Sigma$ with dual basis $v_1^*,\ldots, v_{n-1}^*$ and such that $F(v_1\wedge\ldots\wedge v_{n-1})=1$. If the limit density
 \[ \omega :=\lim_{t\ra 0} \frac{\phi_t^*F-F}{t} \]
 is determined then the jacobian formula is saying that
\[ \int_{\Sigma} d\omega=\int_{\Sigma}\frac{\omega (v_1\wedge\ldots \wedge v_{n-1})}{F(v_1\wedge\ldots \wedge v_{n-1})}~dF.\]
 
  We therefore need to compute $\omega(v_1\wedge\ldots\wedge v_{n-1})$. We have:
 \begin{equation}\label{eq3} \frac{d}{dt}\biggr|_{t=0} F(\Lambda^{n-1}d\phi_t(v_1\wedge\ldots \wedge v_{n-1}))=DF_{v_1\wedge\ldots \wedge v_{n-1}}\left(\frac{d}{dt}\biggr|_{t=0}d\phi_t(v_1)\wedge\ldots \wedge d\phi_t(v_{n-1})\right)
 \end{equation}
 We obviously have 
 \[\frac{d}{dt} d_p\phi_t(v_i)\bigr|_{t=0}=d_pX(v_i).
 \]
 Consequently we get that (\ref{eq3}) equals
 \[DF_{v_1\wedge\ldots \wedge v_{n-1}}\left(\sum_{i=1}^{n-1} v_1\wedge \ldots \wedge dX(v_i)\wedge \ldots \wedge v_{n-1} \right)
 \]
 We recall (see (\ref{Fh}) and Proposition \ref{altcodens}) that in the presence of an inner product the following commutative diagram holds:
 \[ \xymatrix{ \Lambda^{n-1} V \ar[rr]^{A} \ar[dr]_{F} & & V\ar[ld]^{ h} \\ & \bR &}
 \]
 where $h$ is the support function of the Wulff shape $W_F$ and $A$ is the Hodge isomorphism with $w:=A(v_1\wedge\ldots \wedge v_{n-1})$ uniquely identified by the relation:
 \[ \langle w,x\rangle=\Omega(x\wedge v_1\ldots \wedge v_{n-1})\quad\quad \forall x\in V.
 \]
 In other words, $A(v_1\wedge\ldots\wedge v_{n-1})$ is a positive multiple of the oriented unit normal (with respect to the inner product) to the hyperplane spanned by $v_1,\ldots, v_{n-1} $.  Since $A$ is linear we get
 \[  DF_{v_1\wedge\ldots \wedge v_{n-1}}( v_1\wedge \ldots \wedge dX(v_i)\wedge \ldots \wedge v_{n-1})=Dh_{A(v_1\wedge\ldots \wedge v_{n-1})}(A(v_1\wedge\ldots \wedge dX(v_i)\wedge\ldots \wedge v_{n-1}))
 \]
 Now, since $h$ is $1$-homogeneous the gradient of $h$ along any ray is constant and it coincides with the anisotropic normal $n$ to the orthogonal hyperplane of the ray (Corollary 1.7.3 in \cite{Sc}). 
 
 It follows that
 \[ Dh_{A(v_1\wedge\ldots \wedge v_{n-1})}(A(v_1\wedge\ldots \wedge dX(v_i)\wedge\ldots \wedge v_{n-1}))=\langle n, A(v_1\wedge \ldots dX(v_i)\wedge\ldots \wedge v_{n-1} ) \rangle=\]\[=\Omega(n\wedge v_1\wedge\ldots dX(v_i)\wedge \ldots v_{n-1})=v_i^*(P^n_{T\Sigma}(dX(v_i))).
\]
To see why the last equality holds just write $dX(v_i)=an+\sum_j b_j v_j$ and notice that both sides equal $b_i$. One  makes use at this point of the multiplication property (\ref{multprop}) and of $F(v_1\wedge\ldots \wedge v_{n-1})=1$. Finally, notice that by definition
\[ \sum_{i=1}^{n-1}v_i^*(P^n_{T\Sigma}(dX(v_i)))=\tr \nabla^{F}X=\diver_{\Sigma}^F(X).
\]
\end{proof}
Note that according to our definitions if $X=\psi n$ then 
\[ \diver_{\Sigma}^F{X}=\psi\tr S^F.
\]
This result appears also in \cite{An}, Sec. 2.6. It justifies why it is appropriate to call  $\tr S^F$ or $\frac{1}{n}\tr S^F$ the anisotropic \emph{ mean curvature}.
 It turns out that just like in the Riemannian case the variations in the tangent directions do not matter since we have a variant of the divergence theorem,  consequence of the following form of Stokes Theorem.
\begin{theorem}[Stokes] Let $M$ be an $n$-dimensional oriented $C^1$ manifold with boundary $\partial M$. Let $\omega$ be a top degree differential form and $X$ a vector field. Then
\[ \int_M L_X\omega=\int_{\partial M} \iota_X\omega,
\] 
where $L_X$ is the Lie derivative and $\iota_X$ is the contraction.
\end{theorem}
\begin{proof} Use the standard Stokes theorem for the $n-1$ form $\iota_X(\omega)$ together with the relation:
\[ d(\iota_x(\omega))=L_X\omega,
\]
which is nothing else but Cartan's homotopy formula for $n$-forms.
\end{proof}
\begin{remark} The standard Stokes follows from the above formulation by noting that each degree $n-1$ form $\eta$ can be written as $\iota_{X}\omega$ for some $n$-form $\omega$ and a vector field $X$. For example, if $\omega$ is the volume form induced by a Riemannian metric then $X$ can be taken to be the (metric) dual to the Hodge dual of $\eta$.
\end{remark}

We can rewrite the above result as a divergence theorem as follows. Let $\diver^{\omega}{X}$ the divergence of a vector field with respect to a differential form be defined at the points where $\omega$ does not vanish by $L_{X}\omega=\diver^{\omega}{X}\cdot \omega$. On all other points, $\diver^{\omega}{X}$ can be defined arbitrarily. Let $\eta$ be a \emph{non-vanishing} differential form of degree $n-1$ on $\partial M$. Then
\[ \int_{M}\diver^{\omega}(X)\cdot\omega=\int_{\partial M} \omega(X\wedge \ovra{\eta})\cdot\eta, \]
where $\ovra{\eta}$ is the dual to $\eta$ on $\partial M$. This is because $\iota_X\omega=\omega(X\wedge \ovra{\eta})\cdot\eta$. Notice that this formula can be applied with whatever choice of the definition of volume one prefers on $M$ and on $\partial M$. It is easy to see that in the Riemannian case when $\omega=\dvol_M$ and $\eta=\dvol_{\partial M}$ one has $\omega(X\wedge \ovra{\eta})=\langle X,\nu_{\partial M}\rangle$.

\begin{corollary}\label{vfdecomp} Let $X=\psi n+X^{\top}$ be the decomposition of  the variation vector into its anisotropic normal and its tangent components. Then
\[ \frac{\partial \mathscr{F}}{\partial X}(\Sigma)=\int_{\Sigma}\psi\tr{S^F}~dF.
\]
\end{corollary}
\begin{proof} The density $\omega$ from the proof of Theorem \ref{firstvar} is in fact $L_XF$ restricted to $\Sigma$ and $\int_{\Sigma} L_{X^{\top}}F$ vanishes due to the divergence theorem since $\Sigma$ is closed.
\end{proof}
\begin{remark} We take a look at the computations made in \cite{P} with our notations. Let us first notice that as a consequence of the $1$- homogeneity of the support function $h$ the following relation on the unit sphere $S^{n}\subset V$ holds in the presence of an inner product $V$
\[ n(x)=h(x)\cdot x +\nabla_x (h\bigr|_{S^{n}}).
\]
The image of this map determines  the Wulff shape. Moreover, it delivers the following relation between the anisotropic Gauss map and the usual Gauss map of a hypersurface $\Sigma$.
\[ n=h(\nu)\cdot \nu+\nabla_{\nu}(h\bigr|_{S^{n}}).
\]
It is not hard to check that the first variation in the direction of a vector field $X=\psi \cdot n=\psi h(\nu) \cdot  \nu+\psi\nabla_{\nu}(h\bigr|_{S^{n-1}})$ is, according to Palmer equal to
\[ \int_{\Sigma} \psi h(\nu) \tr[ P^{\nu}_{T\Sigma}(dn(\cdot))]~d\mathcal{H}=\int_{\Sigma} \psi \tr[ P^{\nu}_{T\Sigma}(dn(\cdot))]~dF,
\]
where $P^{\nu}_{T\Sigma}$ is the \emph{orthogonal projection} along $\nu$ onto $T\Sigma$. Comparing this with Corollary \ref{vfdecomp} we conclude that
\[ \tr[ P^{\nu}_{T\Sigma}(dn(\cdot))]=\tr[ P^{n}_{T\Sigma}(dn(\cdot))],
\]
must hold at every point.
\end{remark}

\end{document}